\documentclass[12pt,twoside,reqno]{amsart}
\usepackage{amsfonts}
\usepackage{amssymb}
\usepackage[notcite,notref]{%showkeys
}

\numberwithin{equation}{section}

\newtheorem{teo}{Theorem}[section]

\newtheorem{lem}[teo]{Lemma}

\theoremstyle{definition}

\newtheorem{defn}[teo]{Definition}

\newtheorem{rem}[teo]{Remark}

\makeatletter
\@namedef{subjclassname@2010}{%
  \textup{2010} Mathematics Subject Classification}
\makeatother

\def\a{\alpha}
\def\b{\beta}
\def\l{\lambda }
\def\g{\gamma }
\def\o{\omega}
\def\R{\mathbb{R}}

\def\N{\mathbb{N}}
\def\d{\delta}

\def\s{\sigma}

\begin{document}

\title[bounded $\Lambda$-variation and integral smoothness]{On functions of bounded $\Lambda$-variation and integral smoothness}
\author{M. Lind}

\address{Department of Mathematics,
Karlstad University,
Universitetsgatan 2,
651 88 Karlstad,
SWEDEN}

\begin{abstract}
We obtain a necessary and sufficient condition for embeddings of integral Lipschitz classes ${\rm Lip}(\a;p)$ into classes $\Lambda BV$ of functions of bounded $\Lambda$-variation.
\end{abstract}


\maketitle

\section{Introduction}
Jordan's classical concept of bounded variation has been extended in several directions. One well-known generalization is the notion of functions of bounded $p$-variation, due to Wiener.

Waterman \cite{Wat1,Wat2} extended the class of functions of bounded variation in a different way.
Let $f$ be a 1-periodic function on the real line. For any interval $I=[a,b]$, we set $f(I)=f(b)-f(a)$.
Denote by $\mathcal{S}$ the collection of all positive and nondecreasing sequences $\Lambda=\{\l_n\}$ such that $\l_n\rightarrow\infty$ and
$$
\sum_{n=1}^\infty\frac{1}{\l_n}=\infty.
$$
Let $\Lambda=\{\l_n\}\in\mathcal{S}$, a function $f$ is said to be of {\it bounded $\Lambda$-variation} if
\begin{equation}
\nonumber
v_\Lambda(f)=\sup_{\mathcal{I}}\sum_{n=1}^\infty\frac{|f(I_n)|}{\l_n}<\infty,
\end{equation}
where the supremum is taken over all sequences $\mathcal{I}=\{I_n\}$ of nonoverlapping intervals contained in a period. The class of functions of bounded $\Lambda$-variation is denoted $\Lambda BV$. For a discussion on the origin of classes $\Lambda BV$, see \cite{Wat3}. Observe also that all functions of $\Lambda BV$ are bounded.

Denote by $L^p~~(1\le p<\infty)$ the class of all 1-periodic measurable functions $f$ such that
$$
\|f\|_p=\left(\int_0^1|f(x)|^pdx\right)^{1/p}<\infty.
$$
For $f\in L^p~~(1\le p<\infty)$, the {\it $L^p$-modulus of continuity} of $f$ is given by
$$
\o(f;\d)_p=\sup_{0\le h\le\d}\left(\int_0^1|f(x+h)-f(x)|^pdx\right)^{1/p}\quad(0\le\d\le1).
$$
Let $\o$ be any nondecreasing,
continuous and subadditive functions defined on $[0,1]$ with $\o(0)=0$.
For $1\le p<\infty$, set
$$
H_p^\o=\{f\in L^p:\o(f;\d)_p=O(\o(\d))\}.
$$
In the particular case $\o(\d)=\d^\a$, we denote $H_p^\o={\rm Lip}(a;p)$.

Relations between $H_p^\o$ and $\Lambda BV$ have attracted some interest in recent years. Kuprikov \cite{Ku1} obtained a sharp estimate of the $L^p$-modulus of continuity ($1\le p<\infty)$ of a function in terms of its $\Lambda$-variation (for $p=1$, such estimates were first obtained in \cite{ShWa} and \cite{SWa}). Afterwards, Goginava \cite{Go1} proved that Kuprikov's estimate leads to the necessary and sufficient condition for the embedding
$$
\Lambda BV\subset H_p^\o\quad (1\le p<\infty).
$$
This result was later generalized in \cite{HoLePrWi}.

H. Wang \cite{Wan1} studied the reverse embedding, that is 
\begin{equation}
\nonumber
H_p^\o\subset\Lambda BV\quad(1\le p<\infty).
\end{equation}
In particular, he observed that a necessary condition for the embedding
\begin{equation}
\label{Paper4Lipschitz}
{\rm Lip}(\a;p)\subset\Lambda BV\quad(1<p<\infty,\;1/p<\a<1)
\end{equation}
is
\begin{equation}
\label{Paper4LipschitzWrong}
\sum_{n=1}^\infty\left(\frac{1}{\l_n}\right)^{1/(1-\a)}<\infty.
\end{equation}
Wang then conjectured that (\ref{Paper4LipschitzWrong}) is also a sufficient  for (\ref{Paper4Lipschitz}) to hold.
We remark that the condition $\a>1/p$ in (\ref{Paper4Lipschitz}) is essential; for $\a\le1/p$, the class ${\rm Lip}(\a;p)$ contains unbounded functions and (\ref{Paper4Lipschitz}) cannot hold.

The main result of this note is the following. 
\begin{teo}
\label{MainTeoIntroduction}
Let $1<p<\infty$ and $1/p<\a<1$, and set
$$
r=\frac{1}{\a-1/p}\quad{\rm and}\quad r'=\frac{1}{1+1/p-\a}.
$$
Then the embedding (\ref{Paper4Lipschitz}) holds if and only if
\begin{equation}
\nonumber
\sum_{n=0}^\infty\left(\sum_{k=2^n}^{2^{n+1}}\left(\frac{1}{k^{\a-1/p}\l_k}\right)^{p'}\right)^{r'/p'}<\infty.
\end{equation}
\end{teo}
Observe that it follows from Theorem \ref{MainTeoIntroduction} that the conjecture of Wang is not true.

Our proof of Theorem \ref{MainTeoIntroduction} makes essential use of estimates of $L^p$-moduli of continuity in terms of moduli of continuity in the space $V_p$ of functions of bounded $p$-variation. 

In the final part of the paper, we show that $V_p$ can be expressed in terms of spaces $\Lambda BV$. For $p=1$, this result was obtained by Perlman \cite{Pe1}, who proved that
$$
V_1=\bigcap_{\Lambda\in\mathcal{S}}\Lambda BV.
$$
We prove that for $1<p<\infty$, there holds
\begin{equation}
\label{Paper4Perlman}
V_p=\bigcap_{\Lambda\in\mathcal{S}_{p'}}\Lambda BV,
\end{equation}
where $\Lambda\in\mathcal{S}_{p'}$ means that the sequence $\Lambda\in\mathcal{S}$ satisfies
$$
\sum_{n=1}^\infty\left(\frac{1}{\l_n}\right)^{p'}<\infty.
$$
In connection to (\ref{Paper4Perlman}), we mention that embeddings between $\Lambda BV$ and other spaces of functions of generalized bounded variation were previously studied in, e.g., \cite{Be1,Ber1, PeWa, PiVe1}.

%%%%%%%%%%%%%%%%%%%%%%%%%%%%%%%%%%%%%%%%%%%%%%%%%%%%%%%%%%%%%%%%%%%%%%%%%%%%%%%%%%%%%%%%%%%%%%%%%%%%%%%%%%%%%%%%%%%%%%%%%%%%%%%%%%%%%%%%%%%%%%%%%%%%%%%%%%%%%%%%%%%%%%%%%%%%%%
\section{Auxiliary results}
Let $f$ be a 1-periodic function on the real line and $1\le p<\infty$. Below, $\mathcal{I}=\{I_n\}$ will always denote a sequence of nonoverlapping intervals contained in a period. Given any such $\mathcal{I}=\{I_n\}$, set
$$
v_p(f;\mathcal{I})=\left(\sum_{n=1}^\infty|f(I_n)|^p\right)^{1/p}.
$$
The function $f$ is said to be of \emph{bounded $p$-variation} if 
\begin{equation}
\nonumber
v_p(f)=\sup_{\mathcal{I}}v_p(f;\mathcal{I})<\infty,
\end{equation}
where the supremum is taken over all sequences $\mathcal{I}$. For $p=1$, this definition was given by Jordan, and for $p>1$ by Wiener \cite{Wi1}.

Given any $\mathcal{I}=\{I_n\}$, we also set
$\|\mathcal{I}\|=\sup_n|I_n|$, where $|I|$ denotes the length of the interval $I$. Following Terehin \cite{Te1}, we define the {\it modulus of $p$-continuity} for $1<p<\infty$ by
\begin{equation}
\nonumber
\o_{1-1/p}(f;\d)=\sup_{\|\mathcal{I}\|\le\d}v_p(f;\mathcal{I})\quad(0<\d\le1),
\end{equation}
where the supremum is taken over all $\mathcal{I}$ with $\|\mathcal{I}\|\le\d$. Note that $\o_{1-1/p}(f;1)=v_p(f)$, the $p$-variation of $f$.

For $p>1$, we may have that $\lim_{\d\rightarrow0}\o_{1-1/p}(f;\d)=0$
for nontrivial functions. Such functions are called {\it $p$-continuous}. 
It is not difficult to show that the best order of decay of the modulus of $p$-continuity is $\o_{1-1/p}(f;\d)=O(\d^{1/p'}),~~p'=p/(p-1)$. Moreover, this rate of decay is attained for functions in $W_p^1$, the class of 1-periodic absolutely continuous functions $f$ such that $f'\in L^p$. Indeed, it is a simple application of H\"{o}lder's inequality to show that
\begin{equation}
\label{Paper4ModIneq}
\o_{1-1/p}(f;\d)\le\|f'\|_p\d^{1/p'}\quad(0\le\d\le1).
\end{equation}
Conversely, it was shown in \cite{Te1} that if $f$ satisfies $\o_{1-1/p}(f;\d)=O(\d^{1/p'})$, then $f\in W_p^1$.

Let $1<p<\infty$ and $0<\a\le1$. Define
\begin{equation}
\label{Paper4LipschitzNorm}
\|f\|_{{\rm Lip}(\a;p)}=\sup_{\d>0}\frac{\o(f;\d)_p}{\d^\a},
\end{equation}

The next observation will be useful for us.
Let $1<p<\infty$ and $1/p<\a\le1$. Then a function $f\in{\rm Lip}(\a;p)$ can be modified on set of measure 0 to be continuous. 
Moreover, Terehin \cite[Corollary 1]{Te2} proved that there exists a constant $c_{p,\a}>0$ such that for the modified function $\bar{f}$, 
\begin{equation}
\label{Paper4LipschitzEquivalence1}
c_{p,\a}\sup_{\d>0}\frac{\o_{1-1/p}(\bar{f};\d)}{\d^{\a-1/p}}\le\|f\|_{{\rm Lip}(\a;p)}\le\sup_{\d>0}\frac{\o_{1-1/p}(\bar{f};\d)}{\d^{\a-1/p}}.
\end{equation}
Thus, if $p>1$, $1/p<\a\le1$ and $f$ is a continuous 1-periodic function, then
\begin{equation}
\label{Paper4LipschitzEquivalence2}
\o(f;\d)_p=O(\d^\a)\quad{\rm if\;\;and\;\;only\;\;if}\quad\o_{1-1/p}(f;\d)=O(\d^{\a-1/p}).
\end{equation}

We shall also use the following construction.
\begin{defn}
\label{constrDef}
Let $I=[a,b]\subset[0,1]$ be an interval, $N\in\N$ and ${\bf H}=(H_0,H_1,...,H_{N-1})\in\R^N$ be a vector with $H_j\ge0$ for $0\le j\le N-1$. Set $h=(b-a)/N$, $\xi_j=a+jh~~(j=0,1,...,N)$ and $\xi^*_j=a+(j+1/2)h~~(j=0,1,...,N-1)$.
The function $F(x)=F(I,N,{\bf H};x)$ is defined to be the continuous 1-periodic function such that $F(x)=0$ for $x\in[0,1]\setminus I$, $F(\xi_j)=0~~(j=0,1,...,N)$, $F(\xi_j^*)=H_j~~(j=0,1,...,N-1)$, and $F$ is linear on each of the intervals $[\xi_j,\xi_j^*]$ and $[\xi_j^*,\xi_{j+1}]~~(j=0,1,...,N-1)$.
\end{defn}
Thus, the graph of $F$ consists of $N$ isosceles triangles of heights $H_j~~(j=0,...,N-1)$ and bases $h$. It is easy to see that
\begin{equation}
\label{triangle-pvar}
v_p(F)=2^{1/p}\left(\sum_{j=0}^{N-1}H_j^p\right)^{1/p}\quad (1\le p<\infty),
\end{equation}
and
\begin{equation}
\label{triangle-pnorm}
\|F'\|_p=2h^{-1/p'}\left(\sum_{j=0}^{N-1}H_j^p\right)^{1/p}\quad(1\le p<\infty).
\end{equation}

The next lemma is of a known type (cf. \cite{LEP}). In particular, it can be proved in the same way as Lemma 2.4 in \cite{KoLa}.
\begin{lem}
\label{ViktorJavierLemma}
Let $\{\a_k\}\in l^1$ be a sequence of non-negative numbers and let $\theta>1$ and $\gamma>0$. There exists a sequence $\{\b_k\}$ of positive numbers such that
\begin{equation}
\nonumber
\a_k\le\b_k,\quad k\in\N,
\end{equation}
\begin{equation}
\nonumber
\sum_{k=1}^\infty\b_k\le\frac{\theta^{1+\gamma}}{(\theta-1)(\theta^\gamma-1)}\sum_{k=1}^\infty\a_k,
\end{equation}
and
\begin{equation}
\nonumber
\theta^{-\gamma}\le\frac{\b_{k+1}}{\b_k}\le\theta,\quad k\in\N.
\end{equation}
\end{lem}

We shall also use the following Hardy-type inequality (see \cite{Le1}).
\begin{lem}
\label{Paper4LeindlerLemma}
Let $\b>0$ and $1<r<\infty$ be fixed. Let $\{a_k\}$ be a sequence of nonnegative real numbers, and $\{\nu_n\}$ an increasing sequence of positive real numbers with $\nu_0=1$.
Then there exists a constant $c_{\b,r}>0$ such that
\begin{equation}
\label{Paper4Hardy}
\sum_{n=0}^\infty2^{-n\b}\left(\sum_{1\le k\le\nu_n}a_k\right)^{1/r}\le c_{\b,r}\sum_{n=1}^\infty2^{-n\b}\left(\sum_{\nu_{n-1}\le k\le\nu_n}a_k\right)^{1/r}.
\end{equation}
\end{lem}
Finally, we formulate the next well-known result (see, e.g., \cite[Ch.6]{Fo}).
\begin{lem}
\label{dualLemma}
Let $1<p<\infty$. Then $\{x_n\}\in l^p$ if and only if
$$
\sum_{n=1}^\infty \a_nx_n<\infty,
$$
for all $\{\a_n\}\in l^{p'}$. Moreover,
$$
\sup_{\|\{\a_n\}\|_{p'}\le1}\sum_{n=1}^\infty\a_nx_n=\|\{x_n\}\|_p.
$$
\end{lem}

%%%%%%%%%%%%%%%%%%%%%%%%%%%%%%%%%%%%%%%%%%%%%%%%%%%%%%%%%%%%%%%%%%%%%%%%%%%%%%%%%%%%%%%%%%%%%%%%%%%%%%%%%%%%%%%%%%%%%%%%%%%%%%%%%%%%%%%%%%%%%%%%%%%%%%%%%%%%%%%%%%%%%%%%%%%%%

\section{Embedding of Lipschitz classes}
We shall now prove our main results. Recall that $\|f\|_{{\rm Lip}(\a;p)}$ is given by (\ref{Paper4LipschitzNorm}).

\begin{teo}
\label{Paper4MainTeoLip}
Let $\Lambda\in\mathcal{S}$ be given and $1<p<\infty,~~1/p<\a<1$. Set
\begin{equation}
\label{Paper4r-parameter}
r=\frac{1}{\a-1/p}\quad{\rm and}\quad r'=\frac{1}{1+1/p-\a}.
\end{equation}
There exists a constant $c_{p,\a}>0$ depending only on $\a$ and $p$ such that for any $f\in{\rm Lip}(\a;p)$, 
\begin{equation}
\label{Paper4LipschitzEstim}
v_\Lambda(f)\le c_{p,\a}\|f\|_{{\rm Lip}(\a;p)}\left(\sum_{n=0}^\infty\left(\sum_{k=2^n}^{2^{n+1}}\left(\frac{1}{k^{\a-1/p}\l_k}\right)^{p'}\right)^{r'/p'}\right)^{1/r'}.
\end{equation}
\end{teo}
\begin{proof}
In light of (\ref{Paper4LipschitzEquivalence1}), we may without loss of generality assume that
\begin{equation}
\label{Paper4LipschitzNorm2}
\sup_{\d>0}\frac{\o_{1-1/p}(f;\d)}{\d^{\a-1/p}}=1.
\end{equation}
Take an arbitrary sequence $\mathcal{I}=\{I_j\}$ of nonoverlapping intervals contained in a period. Denote 
$$
\s_k(\mathcal{I})=\{j:2^{-k-1}<|I_j|\le2^{-k}\}\quad(k\ge0).
$$
Then we have
$$
V=\sum_{j=1}^\infty\frac{|f(I_j)|}{\l_j}=\sum_{k=0}^\infty\sum_{j\in\s_k(\mathcal{I})}\frac{|f(I_j)|}{\l_j}.
$$
We shall estimate $V$. By H\"{o}lder's inequality, we have
\begin{eqnarray}
\nonumber
V&\le&\sum_{k=0}^\infty\left(\sum_{j\in\s_k(\mathcal{I})}|f(I_j)|^p\right)^{1/p}\left(\sum_{j\in\s_k(\mathcal{I})}\left(\frac{1}{\l_j}\right)^{p'}\right)^{1/p'}\\
\label{Paper4RoughEstim}
&\le&\sum_{k=0}^\infty\o_{1-1/p}(f;2^{-k})\left(\sum_{j\in\s_k(\mathcal{I})}\left(\frac{1}{\l_j}\right)^{p'}\right)^{1/p'}.
\end{eqnarray}
Thus, by (\ref{Paper4LipschitzNorm2}) and (\ref{Paper4RoughEstim})
\begin{equation}
\label{LipschitzRough}
V\le\sum_{k=0}^\infty2^{-k(\a-1/p)}\left(\sum_{j\in\s_k(\mathcal{I})}\left(\frac{1}{\l_j}\right)^{p'}\right)^{1/p'}.
\end{equation}
Let the sequence $\{\d_n\}$ be defined by
$$
{\rm card}\left(\bigcup_{k=0}^n\s_k(\mathcal{I})\right)=2^n\d_n,
$$
where, ${\rm card}(A)$ denotes the number of elements of the finite set $A$. Set also $\d_{-1}=0$.
There exists an $n_0\ge0$ such that $\d_n>0$ for all $n\ge n_0$, and we may assume $n_0=0$. We observe that $\|\{\d_n\}\|_{l^1}\le4$. Indeed, first note that
$$
\sum_{k=0}^\infty2^{-k}{\rm card}(\s_k(\mathcal{I}))\le2\sum_{j=0}^\infty|I_j|\le2.
$$
On the other hand, for $n\ge0$, we have
$$
2^{-n}{\rm card}(\s_n(\mathcal{I}))=\d_n-\d_{n-1}/2.
$$
Whence, for any $N\in\N$, we have
$$
\sum_{n=0}^{N+1}(\d_n-\d_{n-1}/2)=\d_{N+1}+\frac{1}{2}\sum_{n=0}^N\d_n\ge\frac{1}{2}\sum_{n=0}^N\d_n,
$$
and consequently, $\|\{\d_n\}\|_{l^1}\le4$.

Applying Lemma \ref{ViktorJavierLemma} with $\theta=2$ and $\gamma=1/2$ to $\{\d_k\}$ yields a sequence $\{\b_k\}$ such that $\d_k\le\b_k$,
\begin{equation}
\label{quotientEq1}
2^{-1/2}\le\frac{\b_{k+1}}{\b_k}\le2\quad(k\in\N)\quad{\rm and}\quad\|\{\b_k\}\|_{l^1}\le 64.
\end{equation}
Set $\nu_k=2^k\b_k$. By the first relation of (\ref{quotientEq1}), we have
\begin{equation}
\label{quotientEq2}
2\nu_k\le\nu_{k+2}\le 16\nu_k\quad (k\in\N).
\end{equation}
Since ${\rm card}(\s_k(\mathcal{I}))\le 2^k\d_k=\nu_k$, and $\{\l_j\}$ is increasing, we have by (\ref{LipschitzRough})
\begin{equation}
\nonumber
V\le\sum_{k=0}^\infty2^{-k(\a-1/p)}\left(\sum_{1\le j\le\nu_k}\left(\frac{1}{\l_j}\right)^{p'}\right)^{1/p'}.
\end{equation}
Applying (\ref{Paper4Hardy}) to the right-hand side of the previous inequality, we get
$$
V\le c_{p,\a}\sum_{k=1}^\infty2^{-k(\a-1/p)}\left(\sum_{\nu_{k-1}\le j\le\nu_k}\left(\frac{1}{\l_j}\right)^{p'}\right)^{1/p'},
$$
for some constant $c_{p,\a}>0$. Since $2^{-k}=\b_k/\nu_k$, we have
\begin{equation}
\label{Paper4BeforeHolder}
V\le c_{p,\a}\sum_{k=1}^\infty\left(\frac{\b_k}{\nu_k}\right)^{\a-1/p}\left(\sum_{\nu_{k-1}\le j\le\nu_k}\left(\frac{1}{\lambda_j}\right)^{p'}\right)^{1/p'}.
\end{equation}
By using H\"{o}lder's inequality with exponents $r$ and $r'$, and the second inequality of (\ref{quotientEq1}), we estimate the right-hand side of (\ref{Paper4BeforeHolder})
\begin{eqnarray}
\nonumber
V&\le& c_{p,\a}\|\{\b_k^{1/r}\}\|_{l^r}\left(\sum_{k=1}^\infty\nu_k^{-r'(\a-1/p)}\left(\sum_{\nu_{k-1}\le j\le\nu_k}\left(\frac{1}{\l_j}\right)^{p'}\right)^{r'/p'}\right)^{1/r'}\\
\label{Paper4BeforeGrouping}
&\le&64c_{p,\a}\left(\sum_{k=1}^\infty\left(\sum_{\nu_{k-1}\le j\le\nu_k}\left(\frac{1}{j^{\a-1/p}\l_j}\right)^{p'}\right)^{r'/p'}\right)^{1/r'}.
\end{eqnarray}
By collecting the terms of the sum at the right-hand side of (\ref{Paper4BeforeGrouping}) in pairs, and using that $a^q+b^q\le2(a+b)^q$ for any $q\ge0$ and $a,b\ge0$, we get
\begin{equation}
\label{Paper4AfterGrouping}
V\le 128c_{p,\a}\left(\sum_{k=0}^\infty\left(\sum_{\nu_{2k}\le j\le\nu_{2k+2}}\left(\frac{1}{j^{\a-1/p}\l_j}\right)^{p'}\right)^{r'/p'}\right)^{1/r'}.
\end{equation}
For $k\ge0$, we define $m_k\ge0$ as the greatest integer $m$ such that
$$
2^m<\nu_{2k}.
$$
By (\ref{quotientEq2}), we have $2\nu_{2k}\le\nu_{2k+2}$, and thus,
$$
2^{m_k+1}<\nu_{2k+2}.
$$
Consequently,
\begin{equation}
\label{mkEstim}
m_{k+1}\ge m_k+1\quad(k\ge0).
\end{equation}
Further, by (\ref{quotientEq2}), we have $\nu_{2k+2}\le16\nu_{2k}$. Therefore, for all $k\ge0$,
$$
[\nu_{2k},\nu_{2k+2}]\subset [2^{m_k},2^{m_k+5}].
$$
Whence,
$$
\sum_{\nu_{2k}\le j\le\nu_{2k+2}}\left(\frac{1}{j^{\a-1/p}\l_j}\right)^{p'}\le\sum_{j=2^{m_k}}^{2^{m_k+5}}\left(\frac{1}{j^{\a-1/p}\l_j}\right)^{p'}.
$$
Since the terms of the previous sum decrease, it follows that
$$
\sum_{j=2^{m_k}}^{2^{m_k+5}}\left(\frac{1}{j^{\a-1/p}\l_j}\right)^{p'}\le 40\sum_{j=2^{m_k}}^{2^{m_k+1}}\left(\frac{1}{j^{\a-1/p}\l_j}\right)^{p'},
$$
Consequently, by the previous inequality and (\ref{Paper4AfterGrouping}),
$$
V\le c'_{p,\a}\left(\sum_{k=0}^\infty\left(\sum_{j=2^{m_k}}^{2^{m_k+1}}\left(\frac{1}{j^{\a-1/p}\l_j}\right)^{p'}\right)^{r'/p'}\right)^{1/r'},
$$
for some $c_{p,\a}'>0$. By (\ref{mkEstim}), for each $k\ge0$, the intersection of $[2^{m_k},2^{m_k+1}]$ and $[2^{m_{k+1}},2^{m_{k+1}+1}]$ consists of at most one point.
Hence,
$$
V\le c_{p,\a}'\left(\sum_{n=0}^\infty\left(\sum_{j=2^n}^{2^{n+1}}\left(\frac{1}{j^{\a-1/p}\l_j}\right)^{p'}\right)^{r'/p'}\right)^{1/r'}.
$$
This proves (\ref{Paper4LipschitzEstim}).
\end{proof}

The estimate (\ref{Paper4LipschitzEstim}) is sharp in a sense. Namely, we have the following result.
\begin{teo}
\label{Paper4LipschitzSharp}
Let $\Lambda\in\mathcal{S}$ be given, $1<p<\infty,~~1/p<\a<1$ and $r,r'$ be defined by (\ref{Paper4r-parameter}).
Then there exists a function $g$ and constants $c'_{p,\a},c''_{p,\a}>0$ depending only on $\a$ and $p$ such that
\begin{equation}
\label{lipConveq1}
\o_{1-1/p}(g;\d)\le c'_{p,\a}\d^{\a-1/p}\quad(0<\d\le1),
\end{equation}
and
\begin{equation}
\label{lipConveq2}
v_\Lambda(g)\ge c''_{p,\a}\left(\sum_{n=1}^\infty\left(\sum_{k=2^n}^{2^{n+1}}\left(\frac{1}{k^{\a-1/p}\l_k}\right)^{p'}\right)^{r'/p'}\right)^{1/r'}.
\end{equation}
\end{teo}
\begin{proof}
Let $\{\d_n\}\in l^1$ be a fixed but arbitrary positive sequence with $\|\{\d_n\}\|_{l^1}\le1$. Applying Lemma \ref{ViktorJavierLemma} with $\g=1$ and $\theta=3/2$ (the value of $\g$ does not matter, it is only important that $1<\theta<2$) to the sequence $\{\d_n\}$, we obtain a positive sequence $\{\b_n\}$ such that $\d_n\le\b_n~~(n\in\N)$, 
\begin{equation}
\label{Lestim}
\frac{2}{3}<\frac{\b_{n+1}}{\b_n}\le\frac{3}{2}\quad(n\in\N)\quad{\rm and}\quad L=\|\{\b_n\}\|_{l^1}\le9.
\end{equation}
Subdivide the interval $[0,1]$ into non-overlapping intervals $J_n~~(n\in\N)$ with $|J_n|=\b_n/L$. For $n\in\N$, denote
$$
S_n=\left(\sum_{k=2^n}^{2^{n+1}-1}\left(\frac{1}{\l_k}\right)^{p'}\right)^{1/p'},
$$
and
$$
H_k^{(n)}=(2^{-n}\b_n)^{\a-1/p}\l_k^{-1/(p-1)}S_n^{-p'/p}\quad{\rm for}\quad 2^{n}\le k\le2^{n+1}-1.
$$
Let also ${\bf H}_n=(H_{2^n}^{(n)},H_{2^{n}+1}^{(n)},...,H_{2^{n+1}-1}^{(n)})\in \R^{2^n}$. Put $F_n(x)=F(J_n,2^n,{\bf H}_n;x)$ (see Definition \ref{constrDef}), and
$$
g(x)=\sum_{n=1}^\infty F_n(x).
$$
It is clear that
\begin{equation}
\label{LambdaVar1}
v_\Lambda(g)\ge2\sum_{n=1}^\infty\sum_{k=2^n}^{2^{n+1}-1}\frac{H_k^{(n)}}{\l_k}.
\end{equation}
On the other hand, 
\begin{equation}
\nonumber
\sum_{k=2^n}^{2^{n+1}-1}\frac{H_k^{(n)}}{\l_k}=(2^{-n}\b_n)^{\a-1/p}S_n.
\end{equation}
Thus, since $\d_k\le\b_k$ for $k\in\N$, we have
\begin{eqnarray}
\nonumber
\sum_{n=1}^\infty\sum_{k=2^n}^{2^{n+1}-1}\frac{H_k^{(n)}}{\l_k}&=&\sum_{n=1}^\infty\b_n^{\a-1/p}\left(2^{-np'(\a-1/p)}\sum_{k=2^n}^{2^{n+1}-1}\left(\frac{1}{\l_k}\right)^{p'}\right)^{1/p'}\\
\nonumber
&\ge&2^{-1/p+\a}\sum_{n=1}^\infty\d_n^{\a-1/p}\left(\sum_{k=2^n}^{2^{n+1}-1}\left(\frac{1}{k^{\a-1/p}\l_k}\right)^{p'}\right)^{1/p'}.
\end{eqnarray}
By the previous inequality and (\ref{LambdaVar1}),
\begin{equation}
\label{Paper4LambdaVarEstim}
v_\Lambda(g)\ge2^{-1/p+\a}\sum_{n=1}^\infty\d_n^{\a-1/p}\left(\sum_{k=2^n}^{2^{n+1}}\left(\frac{1}{k^{\a-1/p}\l_k}\right)^{p'}\right)^{1/p'}.
\end{equation}

We proceed to estimate $\o_{1-1/p}(g;\d)$. By the first relation of (\ref{Lestim}), we have
$$
\frac{1}{3}\le\frac{2^{-n-1}\b_{n+1}}{2^{-n}\b_n}\le\frac{3}{4}<1.
$$
In particular, the sequence $\{2^{-n}\b_n\}$ is strictly decreasing and  $2^{-n}\b_n\rightarrow0$ as $n\rightarrow\infty$. Fix $0<\d\le1$.
If $\d>2^{-1}\b_1$, then we set $m=0$. Otherwise, define $m\in\N$ to be the unique natural number such that
$$
2^{-m-1}\b_{m+1}<\d\le 2^{-m}\b_m.
$$
By (\ref{Paper4ModIneq}), we have
\begin{equation}
\label{Paper4ModIneq2}
\o_{1-1/p}(g;\d)\le\d^{1/p'}\sum_{n=1}^m\|F'_n\|_p+\sum_{n=m+1}^\infty v_p(F_n).
\end{equation}
(The first sum is taken as zero if $m=0$).
We shall estimate the terms at the right-hand side of (\ref{Paper4ModIneq2}).
It follows from (\ref{triangle-pvar}) and the definition of $H_k^{(n)}$ that
\begin{equation}
\label{hkeq2}
v_p(F_n)=2^{1/p}\left(\sum_{k=2^n}^{2^{n+1}-1}(H_k^{(n)})^p\right)^{1/p}=2^{1/p}(2^{-n}\b_n)^{\a-1/p}.
\end{equation}
Further, by (\ref{triangle-pnorm}),
\begin{eqnarray}
\nonumber
\|F_n'\|_p&=&2\left(\frac{\b_n/L}{2^n}\right)^{-1/p'}\left(\sum_{k=2^n}^{2^{n+1}-1}(H_k^{(n)})^p\right)^{1/p}\\
\label{hkeq3}
&=&2L^{1/p'}(2^{-n}\b_n)^{\a-1}.
\end{eqnarray}
By the estimate $L\le9$, (\ref{Paper4ModIneq2}), (\ref{hkeq2}) and (\ref{hkeq3}),
\begin{eqnarray}
\nonumber
\lefteqn{\o_{1-1/p}(g;\d)\le}\\
&&\le 18\d^{1/p'}\sum_{n=1}^m(2^{-n}\b_n)^{\a-1}+2^{1/p}\sum_{n=m+1}^\infty(2^{-n}\b_n)^{\a-1/p}.
\end{eqnarray}
Since
$$
\left(\frac{2^{-n+1}\b_{n-1}}{2^{-n}\b_n}\right)^{\a-1}=\left(\frac{2\b_{n-1}}{\b_n}\right)^{\a-1}=\left(\frac{\b_n}{2\b_{n-1}}\right)^{1-\a}\le\left(\frac{3}{4}\right)^{1-\a}<1,
$$
we get
\begin{eqnarray}
\nonumber
\sum_{n=1}^m(2^{-n}\b_n)^{\a-1}&\le&(2^{-m}\b_m)^{\a-1}\sum_{n=0}^\infty\left(\frac{3}{4}\right)^{n(1-\a)}=c_\a(2^{-m}\b_m)^{\a-1}\\
\label{Paper4Lip1}
&\le&c_\a\d^{\a-1}.
\end{eqnarray}
Similarly, 
\begin{eqnarray}
\nonumber
\sum_{n=m+1}^\infty(2^{-n}\b_n)^{\a-1/p}&\le& (2^{-m-1}\b_{m+1})^{\a-1/p}\sum_{n=0}^\infty\left(\frac{3}{4}\right)^{n(\a-1/p)}\\
\label{Paper4Lip2}
&\le&c_{p,\a}\d^{\a-1/p}.
\end{eqnarray}
Thus, by (\ref{Paper4ModIneq2}), (\ref{Paper4Lip1}) and (\ref{Paper4Lip2}),
\begin{equation}
\label{Paper4ModIneq3}
\o_{1-1/p}(g;\d)\le c'_{p,\a}\d^{\a-1/p}\quad(0<\d\le1).
\end{equation}
Denote
$$
L_n=\left(\sum_{k=2^n}^{2^{n+1}}\left(\frac{1}{k^{\a-1/p}\l_k}\right)^{p'}\right)^{1/p'}.
$$
Clearly, $\{\d_n\}\in l^1$ is equivalent to $\{\d_n^{\a-1/p}\}\in l^r$.
By Lemma \ref{dualLemma}, we can choose $\{\d_n\}\in l^1$ such that
\begin{equation}
\label{duality1}
\sum_{n=1}^\infty\d_n^{\a-1/p}L_n\ge\frac{1}{2}\left(\sum_{n=1}^\infty L_n^{r'}\right)^{1/r'}.
\end{equation}
If $\{L_n\}\notin l^{r'}$, then we must interpret (\ref{duality1}) in the sense that we may choose $\{\d_n\}\in l^1$ such that the left-hand side of (\ref{duality1}) is infinite. In any case, the function $g$ constructed above with this choice of $\{\d_n\}$ satisfies (\ref{lipConveq1}) and (\ref{lipConveq2}), by (\ref{Paper4ModIneq3}), (\ref{Paper4LambdaVarEstim}) and (\ref{duality1}).
\end{proof}
Clearly, it follows from Theorem \ref{Paper4LipschitzSharp} that the conjecture of Wang is not true (that is, (\ref{Paper4LipschitzWrong}) is not sufficient for the embedding (\ref{Paper4Lipschitz})).

Combining Theorems \ref{Paper4MainTeoLip} and \ref{Paper4LipschitzSharp}, we obtain Theorem \ref{MainTeoIntroduction}.

\begin{rem}
For $1\le p<\infty,~~\a=1$, we have ${\rm Lip}(1;p)=W_p^1$. It is easy to show that the embedding $W_p^1\subset\Lambda BV$ holds for all sequences $\Lambda\in\mathcal{S}$.
\end{rem}

%%%%%%%%%%%%%%%%%%%%%%%%%%%%%%%%%%%%%%%%%%%%%%%%%%%%%%%%%%%%%%%%%%%%%%%%%%%%%%%%%%%%%%%%%%%%%%%%%%%%%%%%%%%%%%%%%%%%%%%%%%%%%%%%%%%%%%%%%%%%%%%%%%%%%%%%%%%%%%%%%%%%%%%%%%%%
\section{A Perlman-type theorem}
As mentioned in the introduction, Perlman \cite{Pe1} showed that
\begin{equation}
\nonumber
V_1=\bigcap_{\Lambda\in\mathcal{S}}\Lambda BV.
\end{equation}

Recall that $\mathcal{S}_{p'}$ denotes the class of all sequences $\Lambda=\{\l_n\}\in\mathcal{S}$ such that
$$
\sum_{n=1}^\infty\left(\frac{1}{\l_n}\right)^{p'}<\infty.
$$
In this section, we prove the following complement to Perlman's theorem.
\begin{teo}
\label{perlmanTeo}
Let $1<p<\infty$. Then
\begin{equation}
\nonumber
V_p=\bigcap_{\Lambda\in\mathcal{S}_{p'}}\Lambda BV.
\end{equation}
\end{teo}
\begin{proof}
Let $f$ be a given function and $\{I_n\}$ an arbitrary sequence of nonoverlapping intervals contained in a period. Applying H\"{o}lder's inequality, we have
$$
\sum_{n=1}^\infty\frac{|f(I_n)|}{\l_n}\le v_p(f)\left(\sum_{n=1}^\infty\left(\frac{1}{\l_n}\right)^{p'}\right)^{1/p'}.
$$
Thus, if $\Lambda\in\mathcal{S}_{p'}$, then $V_p\subset\Lambda BV$. Whence,
$$
V_p\subset\bigcap_{\Lambda\in\mathcal{S}_{p'}}\Lambda BV.
$$
Let now $f$ be a bounded function with $f\notin V_p$. Then there exists a sequence $\{J_n\}$ of nonoverlapping intervals contained in a period such that
$$
\sum_{n=1}^\infty|f(J_n)|^p=\infty.
$$
Since $\{|f(J_n)|\}\notin l^p$, there exists $\{\a_n\}\in l^{p'}$ such that 
$$
\sum_{n=1}^\infty\a_n|f(J_n)|=\infty,
$$ 
by Lemma \ref{dualLemma}. We may assume that $\a_n>0$ for all $n\in\N$ and that $\{|f(J_n)|\}$ is ordered nonincreasingly. Let $\{\a_n^*\}$ be the nonincreasing rearrangement of $\{\a_n\}$, set $\l_n=1/\a_n^*$ and $\Lambda=\{\l_n\}$. Since $\{|f(J_n)|\}$ is nonincreasing, we have
\begin{equation}
\label{perlman3}
\sum_{n=1}^\infty\frac{|f(J_n)|}{\l_n}=\sum_{n=1}^\infty\a_n^*|f(J_n)|\ge\sum_{n=1}^\infty\a_n|f(J_n)|=\infty,
\end{equation}
whence $f\notin \Lambda BV$. It remains to show that $\Lambda\in\mathcal{S}_{p'}$. Clearly $\Lambda$ is a positive and nondecreasing sequence. Moreover, $|f(I)|\le 2\|f\|_\infty$ for any interval. Therefore,
$$
\sum_{n=1}^\infty\frac{1}{\l_n}\ge\frac{1}{2\|f\|_\infty}\sum_{n=1}\frac{|f(J_n)|}{\l_n}=\infty,
$$
by (\ref{perlman3}). Whence, $\Lambda\in\mathcal{S}$. Furthermore, since $\{\a_n\}\in l^{p'}$,
$$
\sum_{n=1}^\infty\left(\frac{1}{\l_n}\right)^{p'}=\sum_{n=1}^\infty(\a_n^*)^{p'}=\sum_{n=1}^\infty\a_n^{p'}<\infty.
$$
Thus, $\{\l_n\}\in\mathcal{S}_{p'}$.
\end{proof}

\begin{rem}
As a generalization of the class $V_p$, L.C. Young \cite{Yo} introduced the class $V_\Phi$ of functions of bounded $\Phi$-variation.
A result similar to Theorem \ref{perlmanTeo} can be proved also for $V_\Phi$.
\end{rem}

\begin{rem}
We can apply Theorem \ref{perlmanTeo} to prove that there is a sequence $\Lambda\in\mathcal{S}$ that satisfies (\ref{Paper4LipschitzWrong}) but still ${\rm Lip}(\a;p)\not\subset\Lambda BV$.

Note first that $1<1/\a<p<\infty$. It was proved in \cite{KoLi1} that there exists a function $f$ such that $\o_{1-1/p}(f;\d)=O(\d^{\a-1/p})$, and at the same time $f\notin V_{1/\a}$.
In light of (\ref{Paper4LipschitzEquivalence2}), this means exactly that there is a function $f\in{\rm Lip}(\a;p)$ such that $f\notin V_{1/a}$.
Theorem \ref{perlmanTeo} states that
\begin{equation}
\label{perlman22}
V_{1/\a}=\bigcap_{\Lambda\in\mathcal{S}_{1/(1-\a)}}\Lambda BV.
\end{equation}
Observe that $\mathcal{S}_{1/(1-\a)}$ is the collection of all sequences in $\mathcal{S}$ that satisfies (\ref{Paper4LipschitzWrong}). 
Since $f\notin V_{1/\a}$, (\ref{perlman22}) implies that for some $\Lambda\in\mathcal{S}_{1/(1-\a)}$, we have $f\notin\Lambda BV$. But since $f\in{\rm Lip}(\a;p)$, we have shown that there exists a $\Lambda$ that satisfies (\ref{Paper4LipschitzWrong}) while the embedding (\ref{Paper4Lipschitz}) does not hold.
\end{rem}

{\sc Acknowledgements.} The author is very grateful to Professor Viktor I. Kolyada, under whose guidance this work was completed.

\end{document}